\documentclass[a4paper, 12pt]{article}
\usepackage{fullpage}
\usepackage{amsmath, amsthm, amssymb, mathrsfs, graphicx, subfigure}
\usepackage{ifthen, url}
\usepackage[usenames]{color}

\usepackage[sort&compress]{natbib}
\bibpunct{(}{)}{;}{a}{}{,}

\theoremstyle{plain}

\newtheorem{lemma}{Lemma}
\newtheorem{proposition}{Proposition}

\theoremstyle{remark}

\theoremstyle{definition}

\theoremstyle{remark}

\theoremstyle{definition}

\newcommand{\FF}{\mathbb{F}}

\newcommand{\YY}{\mathbb{Y}}

\newcommand{\A}{\mathscr{A}}

\newcommand{\Y}{\mathscr{Y}}

\newcommand{\E}{\mathsf{E}}

\newcommand{\prob}{\mathsf{P}}

\newcommand{\eps}{\varepsilon}
\renewcommand{\phi}{\varphi}

\newcommand{\ftrue}{f^\star}

\newcommand{\fhat}{\hat f}

\begin{document}

\title{A note on Bayesian convergence rates under local prior support conditions}
\author{
Ryan Martin \\ 
Department of Mathematics, Statistics, and Computer Science \\ 
University of Illinois at Chicago \\ 
\url{rgmartin@uic.edu} \\
\mbox{} \\
Liang Hong \\
Department of Mathematics \\
Robert Morris University \\
\url{hong@rmu.edu} \\
\mbox{} \\
Stephen G.~Walker \\
Department of Mathematics \\
University of Texas at Austin \\
\url{s.g.walker@math.utexas.edu}
}
\date{\today}

\maketitle

\begin{abstract}
Bounds on Bayesian posterior convergence rates, assuming the prior satisfies both local and global support conditions, are now readily available.  In this paper we explore, in the context of density estimation, Bayesian convergence rates assuming only local prior support conditions.  Our results give optimal rates under minimal conditions using very simple arguments.     

\medskip

\emph{Keywords and phrases:} Density estimation; Hellinger distance; martingale; predictive density; pseudo-posterior.
\end{abstract}

\section{Introduction}
\label{S:intro}

The rate of convergence for Bayesian posterior quantities is helpful in choosing among the long list of seemingly reasonable priors, especially in nonparametric problems.  Since the choice of prior is a fundamental problem in Bayesian analysis, it is no surprise that considerable research efforts have been invested to develop techniques for bounding the rate of convergence.  Key references include \citet{ggv2000}, \citet{shen.wasserman.2001}, \citet{ghosalvaart2001, ghosalvaart2007, ghosalvaart2007a}, and \citet{walker2007}, to name a few.  All of these papers assume that the prior distribution satisfies a local support condition, an extension of the Kullback--Leibler property used in posterior consistency studies \citep[e.g.,][]{ggr1999, bsw1999}.  This local support condition ensures that the prior puts a sufficient amount of mass near the true distribution.  In addition to the local support condition, formal posterior consistency or posterior convergence rate theorems also require that the prior satisfy some global support conditions, though the specific form of these conditions varies from paper to paper.  

In this paper, we focus on convergence rate results that can be obtained assuming only local prior support conditions.  After introducing notation and terminology, we begin in Section~\ref{S:predictive} with an analysis of the behavior of Cesaro averages of Bayesian predictive densities.  Using some very basic argument based on centering, we show that under only local prior support conditions, the Cesaro average convergence rate is arbitrarily close to the optimal rate.  In Section~\ref{S:posterior}, we give conditions such that the posterior probability on sequences of sets which are, in a certain sense, not too close to the true density will vanish.  Such sequences include Hellinger balls not intersecting a collapsing neighborhood of the true density.  Again, we only assume only local support conditions, but since the sets in question are allowed to expand, our Proposition~\ref{prop:rates} strongly suggests a practical posterior convergence rate result.  Section~\ref{S:pseudo} considers a minor modification of the posterior distribution, one that obtains by raising the likelihood to a fractional power before combining with the prior via Bayes theorem.  Our analysis shows that the optimal convergence can be obtained with this so-called pseudo-posterior under only a local prior support condition.  The take-away message from this paper is that, while the existing sufficient conditions for proper Bayesian convergence rate results are somewhat restrictive, desirable results can be established under weaker conditions.  In particular, by removing the global prior support conditions, we show that the posterior distribution is still doing the right things.  Moreover, those somewhat restrictive global prior support conditions, i.e., bounds on metric entropy, etc, often slow down the achievable rate.  We get (near-) optimal rates under minimal conditions using very simple arguments.

\section{Bayesian density estimation}
\label{S:density}

Let $(\YY,\Y)$ be a measurable space, and let $Y_1,\ldots,Y_n$ be independent $\YY$-valued random variables having density $f$ with respect to a $\sigma$-finite measure $\mu$ on $\Y$.  The goal is inference on $f$.  Following the Bayesian approach, let $\FF$ be a subset of all $\mu$-densities $f$, and $\Pi$ a prior distribution supported on $\FF$.  Examples of priors for densities include Dirichlet process mixtures and their variants, Polya trees, Bernstein polynomials, and logistic Gaussian processes.  Then Bayes theorem gives the posterior distribution of $f$, given $Y_1,\ldots,Y_n$:
\begin{equation}
\label{eq:post1}
\Pi_n(A) = \Pi(A \mid Y_1,\ldots,Y_n) = \frac{\int_A \prod_{i=1}^n f(Y_i) \, \Pi(df)}{\int_{\FF} \prod_{i=1}^n f(Y_i) \, \Pi(df)}, \quad A \subseteq \FF.
\end{equation}
The posterior distribution $\Pi_n$ gives a complete probabilistic summary of the information relevant for inference about $f$.  For example, the posterior mean, $\hat f_n = \int f \,\Pi_n(df)$, also known as the predictive density (see Section~\ref{S:predictive}), is a natural estimator of $f$.  

Bayesian convergence results concern the asymptotic behavior of certain functionals of the posterior $\Pi_n$, under the iid $\ftrue$ model, as $n \to \infty$.  Several such results are considered in the upcoming sections, and all can be understood as describing a sense in which the posterior concentrates around $\ftrue$ asymptotically.  However, in order for the posterior $\Pi_n$ to concentrate around $\ftrue$, the prior $\Pi$ should also be sufficiently concentrated around $\ftrue$.  In the posterior consistency literature, the Kullback--Leibler property is the most natural condition \citep{schwartz1965, ggr1999, bsw1999, wu.ghosal.2008, choi.ramamoorthi.2008}.  An obvious extension of the Kullback--Leibler support condition in the more challenging rates problem is as follows.  For a positive vanishing sequence of numbers $\eps_n$, assume that 
\begin{equation}
\label{eq:support1}
\Pi(\{f \in \FF: K(\ftrue, f) \leq \eps_n^2\}) \geq e^{-Cn\eps_n^2}, 
\end{equation}
where $K(\ftrue, f) = \int \log(\ftrue / f) \ftrue \,d\mu$ is the Kullback--Leibler divergence of $f$ from $\ftrue$ and $C > 0$ is a constant.  This will be support condition considered in Section~\ref{S:predictive}, but a stronger condition will be assumed in Sections~\ref{S:posterior} and \ref{S:pseudo}; see, also, \citet{ggv2000}, \citet{shen.wasserman.2001}, and \citet{walker2007}.

\section{Convergence rates for predictive densities}
\label{S:predictive}

Here we investigate the asymptotic behavior of the predictive density $\hat f_n$.  Consistency of $\hat f_n$ was considered by \citet{barron1987, barron1999} and \citet{walker2003, walker2004a}, and here we extend the consistency result to obtain rates of convergence assuming only \eqref{eq:support1}.  

First, we need a bit more notation.  If $\ftrue$ is the true density from which the data $Y_1,\ldots,Y_n$ are observed, it is typical to rewrite the posterior \eqref{eq:post1} as 
\begin{equation}
\label{eq:post2}
\Pi_n(A) = \frac{\int_A R_n(f) \,\Pi(df)}{\int_{\FF} R_n(f) \,\Pi(df)}, \quad A \subseteq \FF, 
\end{equation}
where $R_0(f) \equiv 1$ and $R_n(f) = \prod_{i=1}^n f(Y_i) / \ftrue(Y_i)$, $n \geq 1$.  Write $I_n$ for the denominator in \eqref{eq:post2}.  Then we have the following simple consequence of \eqref{eq:support1}.  

\begin{lemma}
\label{lem:support1}
If $\Pi$ satisfies \eqref{eq:support1}, then $\E(\log I_n) \geq -(C+1) n\eps_n^2$.  
\end{lemma}

\begin{proof}
Let $K_n = \{f: K(\ftrue,f) \leq \eps_n^2\}$ and $\Pi^{K_n}$ the version of $\Pi$ restricted and normalized on $K_n$.  Lower bound $I_n$ by $\Pi(K_n) \int_{K_n} R_n(f) \, \Pi^{K_n}(df)$, which is valid since $R_n(f)$ is non-negative.  Take a log and use Jensen's inequality to get 
\[ \log I_n \geq \log \Pi(K_n) + \int_{K_n} \log R_n(f) \, \Pi^{K_n}(df). \]
Now take expectation with respect to $\ftrue$ and apply Fubini's theorem to get 
\[ \E(\log I_n) \geq \log \Pi(K_n) - n\int_{K_n} K(\ftrue,f) \,\Pi^{K_n}(df). \]
The first term is $\geq -Cn\eps_n^2$ by assumption, and the second term is $\geq n\eps_n^2$ by the construction of $K_n$.  Therefore, $\E(\log I_n) \geq -(C+1)n\eps_n^2$, proving the claim.  
\end{proof}

For convergence of the predictive density $\fhat_n$, the key observation is that 
\[ I_i / I_{i-1} = \fhat_{i-1}(Y_i) / \ftrue(Y_i), \quad i \geq 1 \qquad [I_0 \equiv 1]. \] 
If $\Y_i$ is the $\sigma$-algebra generated by the data $Y_1,\ldots,Y_i$, then 
\[ \E\{ \log(I_i / I_{i-1}) \mid \Y_{i-1} \} = -K(\ftrue, \fhat_{i-1}). \]
Then $X_i := \log(I_i / I_{i-1}) + K(\ftrue,\fhat_{i-1})$ forms a martingale difference sequence and, in particular $\E(X_i) = 0$ for all $i \geq 1$.  Then we have the following elementary result.

\begin{proposition}
\label{prop:pred.rates}
Suppose $\Pi$ satisfies \eqref{eq:support1} for a constant $C > 0$ and sequence $\eps_n$ with $\eps_n \to 0$ and $n\eps_n^2 \to \infty$.  Then $n^{-1} \sum_{i=1}^n \E\{ K(\ftrue, \fhat_{i-1}) \} \leq (C+1)\eps_n^2$ and, furthermore, if $\bar f_n = n^{-1} \sum_{i=1}^n \fhat_{i-1}$, then $\E\{K(\ftrue, \bar f_n)\} \leq (C+1)\eps_n^2$.  
\end{proposition}

\begin{proof}
By construction, 
\[ 0 = \E(X_1+\cdots+X_n) = \E(\log I_n) + \sum_{i=1}^n \E\{K(\ftrue, \fhat_{i-1})\}. \]
Since $\E(\log I_n) \geq -(C+1)n\eps_n^2$ by Lemma~\ref{lem:support1}, the right-most term must be $\leq (C+1)n\eps_n^2$.  The part with $\bar f_n$ follows from this and convexity of $K$.  
\end{proof}

By Markov's inequality, the in-probability rate for $K(\ftrue, \bar f_n) \to 0$ is arbitrarily close to $\eps_n^2$.  Next, let $H$ denote the Hellinger distance on $\FF$, given by $H(f,g)^2 = \int (f^{1/2}-g^{1/2})^2 \,d\mu$, and write $h=H^2/2$.  Then the same in-probability rate holds for $h(\ftrue,\bar f_n) \to 0$, so the Hellinger convergence rate of $\bar f_n$ to $\ftrue$ is arbitrarily close to $\eps_n$.  

One might ask if the Hellinger rate of convergence for $\bar f_n$ in Proposition~\ref{prop:pred.rates} extends to the predictive density $\fhat_n$ itself.  A precise result is difficult, but the following heuristics suggest that $\fhat_n$ cannot have a different asymptotic behavior than $\bar f_n$ except under extraordinary circumstances.  Consider a generic positive sequence of numbers $a_n$ such that $n^{-1}\sum_{i=1}^n a_i \to 0$ but $a_n \not\to 0$.  This implies that the $a_n$ sequence must be generally decreasing to zero but have some regularly occurring and significant jumps.  In our case, it would be virtually impossible, especially without knowledge of $\ftrue$, to construct a prior $\Pi$ such that $h(\ftrue, \fhat_n)$ could behave in this unusual way on sets with large probability.  So, based on this understanding, we feel safe extending the rate result to $\fhat_n$.  

The main point of Proposition~\ref{prop:pred.rates} is that the Bayesian can has access to a consistent density estimate ($\bar f_n$ or $\fhat_n$) under only local support conditions, and the rate of convergence is determined by only the prior concentration in \eqref{eq:support1}.  Existing posterior convergence rate results require global support conditions and, in general, yield slower convergence rates; see, e.g., Theorem~2.1 in \citet{ghosalvaart2001} and their Dirichlet process mixture prior application.


\section{Posterior behavior away from $\ftrue$}
\label{S:posterior}

This section explores the behavior of the posterior distribution for sequences of sets $A_n$ in $\FF$ that do not get too close to $\ftrue$.  For this, we require a stronger version of \eqref{eq:support1}.  Let $V(\ftrue,f) = \int \{\log(\ftrue/f)\}^2 \ftrue \,d\mu$ and, for a sequence $\eps_n$ as before, consider 
\begin{equation}
\label{eq:support2}
\Pi(\{f: K(\ftrue,f) \leq \eps_n^2,\, V(\ftrue,f) \leq \eps_n^2\}) \geq e^{-Cn\eps_n^2}.
\end{equation}
This is clearly a stronger condition on $\Pi$ than \eqref{eq:support1}.  The following lemma, an analogue to Lemma~\ref{lem:support1} above, gives an in-probability bound on the denominator $I_n$ in \eqref{eq:post2}; see \citet[][Lemma~8.1]{ggv2000} for a proof.  Here, and in what follows, a statement ``$U_n \leq V_n$ in probability'' means that $U_n \leq V_n$ with probability approaching~1 as $n \to \infty$.  

\begin{lemma}
\label{lem:support2}
Let $I_n = \int R_n(f) \,\Pi(df)$ be the denominator in \eqref{eq:post2}.  If $\Pi$ satisfies \eqref{eq:support2}, then $I_n \geq e^{-cn\eps_n^2}$ in probability for any $c > C+1$.  
\end{lemma}

Let $\fhat_i^{A_n}$ denotes the predictive distribution of $Y_{i+1}$, given $Y_1,\ldots,Y_i$, $i=1,\ldots,n$, when $\Pi_i$ is restricted and normalized to $A_n$.   Let $L_{n,i} = \int_{A_n} R_i(f)\,\Pi(df)$ be the numerator of $\Pi_i(A_n)$ in \eqref{eq:post2}, $i=1,\ldots,n$.  Then it is clear that 
\[ L_{n,i} \,/\, L_{n,i-1} = \fhat_{i-1}^{A_n}(Y_i) \,/\, \ftrue(Y_i), \quad i=1,\ldots,n, \quad [L_{n,0} \equiv \Pi(A_n)]. \]
It is easy to check that $\E\{ (L_{n,i} / L_{n,i-1})^{1/2}-1 \mid \Y_{i-1}\} = -h(\ftrue, \fhat_{i-1}^{A_n})$.  Therefore, $X_{n,i} = (L_{n,i} / L_{n,i-1})^{1/2}-1 + h(\ftrue, \fhat_{i-1}^{A_n})$ forms a martingale difference array.  This martingale representation allows us to prove the following result.  

\begin{proposition}
\label{prop:rates}
For given $\eps_n$, with $\eps_n \to 0$ and $n\eps_n^2 \to \infty$, and $C > 0$, assume that $\Pi$ satisfies \eqref{eq:support2}.  If, for some $\beta \in (0,1/2)$ and $D > (C+1)/2$, 
\begin{equation}
\label{eq:condition1}
\frac1n \sum_{i=1}^n h(\ftrue, \fhat_{i-1}^{A_n}) \geq Dn^{-\beta}, \quad \text{in probability},
\end{equation}
then $\Pi_n(A_n) \leq \Pi(A_n) e^{-\kappa n\delta_n^2}$, in probability, for some $\kappa > 0$, where $\delta_n^2 = n^{-\beta} \wedge \eps_n^2$. 
\end{proposition}

\begin{proof}
For $X_{n,i}$ defined above, and $\Y_{i-1}$ defined in Section~\ref{S:predictive}, the key result is 
\[ \E(X_{n,i} \mid \Y_{i-1}) \leq \int \{ (\fhat_{i-1}^{A_n} / \ftrue)^{1/2} - 1\}^2 \ftrue \,d\mu \leq 2 h(\ftrue, \fhat_{i-1}^{A_n}) \leq 2. \]
Then $M_{n,n} = \sum_{i=1}^n X_{n,i}$ is a square-integrable martingale, with 
\[ \E(M_{n,n}^2) = \sum_{i=1}^n \E(X_{n,i}^2) \leq 2n. \]
Let $\omega_n=n^{1-\beta}$ for the $\beta$ in the statement above.  It follows from Markov's inequality that $M_{n,n}/\omega_n \leq d$ in probability for any $d > 0$.  Moreover, from \eqref{eq:condition1}, we can conclude that 
\[ \frac{M_{n,n}}{\omega_n} \geq \frac{1}{\omega_n} \sum_{i=1}^n \Bigl\{ \Bigl( \frac{L_{n,i}}{L_{n,i-1}} \Bigr)^{1/2}-1\Bigr\} + D, \quad \text{in probability}. \]
Rearranging this inequality, and using the fact that arithmetic means are no smaller than geometric means, and the inequality $\log x \leq x-1$, we get
\[ \frac{1}{2\omega_n}\log \frac{L_{n,n}}{L_{n,0}} \leq \frac{M_{n,n}}{\omega_n} - D. \]
Since $M_{n,n}/\omega_n \leq d$ in probability, for any $d > 0$, we get $L_{n,n} \leq \Pi(A_n)e^{-2(D-d)\omega_n}$ in probability.  Also, from Lemma~\ref{lem:support2}, $I_n \geq e^{-cn\eps_n^2}$ in probability for any $c \in (C+1, 2D)$.  Therefore, 
\[ \Pi_n(A_n) = \frac{L_{n,n}}{I_n} \leq \Pi(A_n) e^{-(2D-2d-c)n\delta_n^2}, \quad \text{in probability}. \]
To complete the proof, take $d$ small enough that $\kappa = 2D-2d-c$ is positive.  
\end{proof}

Often, $\eps_n^2$ will be smaller than $n^{-1/2}$, e.g., \citet{ghosalvaart2001} get $\eps_n^2=(\log n)^2n^{-1}$ in their Dirichlet process mixture setting, in which case $\delta_n^2$ from Proposition~\ref{prop:rates} is exactly $\eps_n^2$.  The boundary, where $\delta_n^2$ switches between $\eps_n^2$ and $n^{-\beta}$, is the case $\eps_n=n^{-1/4}$, which appears, for example, in the context of estimating a smooth density with a log-Brownian motion-type prior \citep{vaart.zanten.2008, castillo2008}.   

The result in Proposition~\ref{prop:rates} applies for sets $A_n$ that do not get too close to $\ftrue$.  Hellinger balls $A_n$ with suitable center and radii would satisfy \eqref{eq:condition1}, but it would certainly hold for other sequences $A_n$.  Except for prior support conditions, the only requirement is that the mean of the posterior $\Pi_n^{A_n}$, restricted and normalized to $A_n$, does not agree with $\ftrue$.  We cannot imagine a reasonable prior, i.e., one without knowledge of $\ftrue$, and a sequence of Hellinger balls $A_n$, sufficiently separated from $\ftrue$, for which \eqref{eq:condition1} might fail. 

As an example, take $A_n \equiv A$ fixed.  Then $\Pi_n(A) \to 0$ if $n^{-(1-\beta)} \sum_{i=1}^n h(\ftrue, \fhat_{i-1}^A)$ is bounded away from zero.  \citet{walker2003} reaches the same conclusion based on the assumption that $h(\ftrue, \fhat_n^A)$ is bounded away from zero.  Since $\beta > 0$, our condition is weaker than Walker's, meaning that $\Pi_n(A) \to 0$ for a wider class of sets $A$. 

It is straightforward to extend Proposition~\ref{prop:rates} to a finite collection of sequences, say, $(A_{nj})$, where $n \geq 1$ and $j=1,\ldots,J$ for fixed finite $J$.  In that case, 
\[ \Pi_n(A_{n1} \cup \cdots \cup A_{nJ}) \leq \sum_{j=1}^J \Pi_n(A_{nj}) \to 0 \quad \text{in probability}. \]
Suppose that the $A_{nj}$'s are Hellinger balls with radius increasing with $n$ and center $f_{nj}$ moving away from $\ftrue$ in such a way that \eqref{eq:condition1} holds for each $j=1,\ldots,J$.  If we take $J$ to be very large, then, in some sense, the union $A_{n1} \cup \cdots \cup A_{nJ}$ of these expanding balls almost fills up the space outside the collapsing neighborhood of $\ftrue$, suggesting that the posterior is concentrating on a Hellinger ball at $\ftrue$ of radius proportional to $\eps_n$.  


\section{Pseudo-posterior convergence rates}
\label{S:pseudo}

The results of the previous two sections are simple and provide some useful insight, but they fall short of giving a formal posterior convergence rate theorem.  However, if a formal convergence rate theorem is the goal, then it can be easily obtained with a slight modification to the construction of the posterior; this modified posterior shall be called a pseudo-posterior.  This technique was first introduced in \citet{walker.hjort.2001}, where pseudo-posterior consistency results were readily obtained.  To our knowledge, pseudo-posterior convergence rates have not been considered in general.  

Specifically, consider a posterior obtained by Bayes theorem based on a one-half fractional power of the likelihood.  That is, let 
\begin{equation}
\label{eq:pseudo}
\tilde \Pi_n(A) = \tilde \Pi(A \mid Y_1,\ldots,Y_n) \propto \int_A L_n(f)^{1/2} \,\Pi(df), \quad A \subseteq \FF, 
\end{equation}
where $L_n(f) = \prod_{i=1}^n f(Y_i)$ is the likelihood.  So, the only difference between $\tilde\Pi_n$ and $\Pi_n$ is the one-half fraction power on the likelihood.  That the pseudo-posterior is proper whenever the corresponding posterior is proper is a simple consequence of Jensen's inequality.  Computation of the pseudo-posterior in the context of density estimation with Dirichlet process mixtures was addressed in \citet{walker.fractional.2013} and other applications are presented in \citet{scott.shively.walker.2013} and \citet{martin.walker.eb}.  

An alternative way to interpret the pseudo-posterior is as an empirical Bayes posterior.  That is, the same pseudo-posterior obtains if one uses the regular likelihood but replaces the prior $\Pi$ with the data-dependent measure $\Gamma_n$ with density $\Gamma_n(df) = L_n(f)^{-1/2} \,\Pi(df)$.  Since the pseudo-posterior also corresponds to a genuine posterior but with a data-dependent prior, it can also be interpreted as an empirical Bayes posterior.  Intuitively, bad posterior behavior occurs when the prior assigns too much weight to densities $f$ that track data too closely \citep{walker.lijoi.prunster.2005a}.  Such densities have high likelihood and, consequently, lower empirical Bayes prior mass, so these bad densities have less of an effect on the behavior of the pseudo-posterior.  

Let $I_n = \int R_n(f)^{1/2} \,\Pi(df)$ be the denominator of the pseudo-posterior.  If $\Pi$ satisfies the support condition \eqref{eq:support2}, then there is a lower bound result for $I_n$, analogous to that in Lemma~\ref{lem:support2}, i.e., $I_n \geq e^{-cn\eps_n^2}$ in probability for all $c > (C+1)/2$.  

\begin{proposition}
\label{prop:pseudo}
Given $\eps_n$ such that $\eps_n \to 0$ and $n\eps_n^2 \to \infty$, assume that the prior $\Pi$ satisfies the support condition \eqref{eq:support2}.  Let $A_n = \{f \in \FF: H(\ftrue, f) > M\eps_n\}$, where $M^2 > (C+1)/2$.  Then the pseudo-posterior satisfies $\tilde \Pi_n(A_n) \to 0$ in probability.  
\end{proposition}

\begin{proof}
Let $U_n = \int_{A_n} R_n(f)^{1/2} \, \Pi(df)$ be the numerator of the pseudo-posterior probability of $A_n$.  As in the proof of Theorem~1 in \citet{walker.hjort.2001}, it is easy to check that $\E(U_n) \leq e^{-M^2n\eps_n^2}$ and, therefore, $U_n \leq e^{-Kn\eps_n^2}$ in probability for all $K < M^2$.  Choose $c$ such that $(C+1)/2 < c < K < M^2$, then $I_n \geq e^{-cn\eps_n^2}$ in probability and, consequently, 
\[ \tilde \Pi_n(A_n) = U_n/I_n \leq e^{-(K-c)n\eps_n^2} \quad \text{in probability}. \]
Since $K > c$, the desired result follows.  
\end{proof}

The key point here is that if getting good posterior convergence rates is the goal, then it can be done very easily with a pseudo-posterior based on any prior that satisfies the support condition \eqref{eq:support2}.  In fact, in some cases, the pseudo-posterior rates, which are determined only by the local prior concentration around $\ftrue$, are faster than the known rates for the genuine posterior.  One example is the Dirichlet process mixture prior in \citet{ghosalvaart2001}, mentioned earlier, and another is the Bernstein polynomial prior in \citet{ghosal2001}: in both cases, the pseudo-posterior rate will be faster by a logarithmic factor.  

It can also be shown that the result in Proposition~\ref{prop:pseudo} hold for any fraction power $\kappa \in (0,1)$, not just for $\kappa=1/2$.  Therefore, by taking $\kappa$ arbitrarily close to 1, the corresponding pseudo-posterior cannot differ too much from the genuine posterior in finite-sample applications and, moreover, the former enjoys an asymptotic convergence rate result under only a local prior support condition, while the latter generally does not.  Of course, from Proposition~\ref{prop:pseudo}, one can also derive Hellinger convergence rate results for the pseudo-posterior mean density as in \citet{ggv2000}, among other things.

\bibliographystyle{apalike}
\bibliography{/Users/rgmartin/Dropbox/Research/mybib}

\end{document}